\documentclass[a4paper,10pt]{amsart}
\usepackage[utf8]{inputenc}
\usepackage{color}
\usepackage{graphicx}

\usepackage{amssymb}
\usepackage{amsfonts}
\usepackage{amsmath}
\usepackage{amsthm}
\newtheorem{proposition}{Proposition} 

\newcommand\beq{\begin{equation}}
\newcommand\eeq{\end{equation}}

\renewcommand{\to}{\rightarrow}

\renewcommand{\det}{\operatorname{det}}
\renewcommand{\div}{\operatorname{div}}

\newcommand{\grad}{\boldsymbol{\nabla}}

\newcommand{\ba}{\boldsymbol{a}} 
\newcommand{\bA}{\boldsymbol{A}} 

\newcommand{\bc}{\boldsymbol{c}}
\newcommand{\bC}{\boldsymbol{C}}

\def\bD{\boldsymbol{D}}

\newcommand{\be}{\boldsymbol{e}}

\newcommand{\boldf}{\boldsymbol{f}}

\newcommand{\bF}{\boldsymbol{F}}

\newcommand{\bI}{\boldsymbol{I}} 

\def\bS{\boldsymbol S}

\newcommand{\bu}{\boldsymbol{u}}

\newcommand{\gbu}{{\grad\bu}}

\newcommand{\bx}{\boldsymbol{x}} 

\newcommand{\bY}{\boldsymbol{Y}}

\newcommand{\bphi}{\boldsymbol{\phi}}

\newcommand{\bsigma}{\boldsymbol{\sigma}}

\newcommand{\btau}{\boldsymbol{\tau}}

\def\Bcal{\mathcal{B}}

\newcommand{\RR}{\mathbb{R}} \newcommand{\R}{\mathbb{R}} 

\newcommand{\bzero}{{\boldsymbol{0}}}

\title{
 A viscoelastic 
 flow model of Maxwell-type \\
 with a symmetric-hyperbolic formulation
}
\author{S\'ebastien Boyaval}
\address{LHSV, Ecole des Ponts, EDF R\&D, Chatou, France, sebastien.boyaval@enpc.fr}
\address{MATHERIALS, Inria, Paris, France}
\subjclass[2010]{76A10; 35L45; 74D10}

\begin{document}

\begin{abstract}
Maxwell models for viscoelastic 
flows are famous 
for their potential 
to unify 
elastic motions of solids with viscous motions of liquids
in the continuum mechanics perspective.
But the usual Maxwell models allow one to define well motions mostly for \emph{one-dimensional} 
flows only.
To define unequivocal \emph{multi-dimensional} viscoelastic flows 
(as solutions to well-posed 
initial-value problems) 
we advocated in [ESAIM:M2AN 55 (2021) 807-831]
an upper-convected Maxwell model for compressible 
flows with a symmetric-hyperbolic formulation.
Here, that model is derived again, with new details.
\end{abstract}




\maketitle


\section{Elastic and viscous motions 
in the continuum perspective}
\label{sec:intro}

First, let us recall seminal systems of PDEs that unequivocally model the 
motions $\bphi_t: \Bcal \to \subset\RR^3$ 
of 
\emph{continuum bodies} 
$\Bcal$ 
on a time range $t\in[0,T)$. 
%
PDEs governing \emph{elastic flows} 
are a starting point 
for all continuum bodies. 
PDEs governing \emph{viscoelastic flows}, 
for liquid 
bodies in particular, shall come next in Section~\ref{sec:maxwell}.

\bigskip

Let us denote $\{x^i, i=1\dots 3\}$ a Cartesian coordinate system for the Euclidean ambiant space $\RR^3$.
Let us assume, for $t\in[0,T)$,
that $\Bcal$ is a manifold equipped with a Cartesian coordinate system $\{a^\alpha, \alpha=1\dots d\}$ ($d\in\{1,2,3\}$),
and that $\bphi_t(\ba \equiv a^\alpha \be_\alpha) =  \phi^i_t(\ba) \be_i$
is a 
bi-Lipshitz function on $\Bcal\ni\ba$. 
%
Given a vector force field $\boldf$ in $\RR^3$, Galilean physics requires
the \emph{deformation gradient} 
$F^i_\alpha := \partial_\alpha \phi^i_t \circ \bphi_t^{-1}$ 
and the 
\emph{velocity} $u^i := \partial_t \phi^i_t \circ \bphi_t^{-1}$,
to satisfy the conservation of linear momentum:
\begin{equation}
\label{eq:momentum_material}
\hat\rho \partial_t (\bu\circ\bphi_t) = \div_{\ba} \bS + \hat\rho (\boldf\circ\bphi_t) \text{ on } \Bcal
\end{equation} 
given a mass-density $\hat\rho(\ba)\ge0$, 
see e.g. \cite{marsden-hughes-2012-mathematical}.
Neglecting heat transfers, the first Piola-Kirchoff stress tensor $\bS(\bF)$ is defined
by an internal energy functional $e(\bF)$: 
\begin{equation}
\label{eq:2Dpiolakirchhof}
S^{i\alpha}  = \hat\rho \partial_{ F^i_\alpha } e \,.
\end{equation}
Then, when $\hat\rho\in\R^+_*$ is constant, 
motions can be unequivocally defined 
by 
solutions 
$$(u^i\circ\bphi_t,F^i_\alpha\circ\bphi_t)\in C_t^0\left([0,T),H^s(\RR^3)^3\times H^s(\RR^3)^{3\times3}\right) \text{ with } s>\frac32 $$
to (\ref{eq:momentum_material}--\ref{eq:2Dpiolakirchhof}) complemented by 
(\ref{eq:deformation_material}--
\ref{eq:cofactor_material}), if (\ref{eq:momentum_material}--\ref{eq:cofactor_material})
defines a \emph{symmetric-hyperbolic} system 
\cite{DafermosBook4}, 
\begin{align}
\label{eq:deformation_material}
\partial_t \left(F^i_\alpha\circ\bphi_t \right)- \partial_\alpha \left( u^i\circ\bphi_t \right)= 0
\\
\label{eq:determinant_material}
\partial_t \left(|F^i_\alpha|\circ\bphi_t \right)- \partial_{\alpha}\left( C^i_\alpha\circ\bphi_t \: u^i\circ\bphi_t \right) = 0
\\ 
\label{eq:cofactor_material}
\partial_t \left( C^i_\alpha\circ\bphi_t \right) + \sigma_{ijk}\sigma_{\alpha\beta\gamma}\partial_\beta \left( F^j_\gamma\circ\bphi_t \: u^k\circ\bphi_t \right) = 0 
\end{align}
denoting $\sigma_{ijk}$ Levi-Civita's symbol.
%
But for physical applications, it is difficult to identify  functionals $e(\bF)$ such that 
(\ref{eq:momentum_material}--\ref{eq:cofactor_material}) 
defines a \emph{symmetric-hyperbolic} system.

In the sequel, 
assuming $\hat\rho\in\R^+_*$ constant, 
we recall how 
one standardly defines $e(\bF)$ for solids and fluids dynamics, 
on considering the determinant $|F^i_\alpha|$ of the deformation gradient (also denoted $|\bF|$ hereafter)
and the cofactor matrix $C^i_\alpha$ 
of $F^i_\alpha$ ($\bC$ in tensor notation) as variables independent of $\bF$.
Next, in Section \ref{sec:maxwell}, we recall with much details the function $e(\bF)$ that we proposed in \cite{Boyaval2021} 
so as to properly define a viscoelastic dynamics of Maxwell type that 
\emph{unifies} solids and fluids. 

\subsection{Polyconvex elastodynamics} 

If $e(\bF)$ 
in \eqref{eq:2Dpiolakirchhof} is \emph{polyconvex}, 
and 
if the initial conditions for 
$(\bu,\bF,|\bF|,\bC)\circ\bphi_t$ 
are given 
by $\left(\partial_t\bphi_t,\nabla_{\ba}\bphi_t,|\nabla_{\ba}\bphi_t|,\mathop{\rm Cof}(\nabla_{\ba}\bphi_t)\right)(t=0)
\in H^s(\RR^3)$ with $s>\frac32$, such that $\nabla_{\ba}\times\bF=\bzero=\div_{\ba}\bC$ holds 
i.e 
\begin{equation}
\label{eq:piola_material}
\sigma_{\alpha\beta\gamma} \partial_{\alpha} F^i_\beta = 0 = \partial_\alpha C^i_\alpha \quad \forall i \,,
\end{equation}
then 
(\ref{eq:momentum_material}--\ref{eq:cofactor_material})
enters the framework of symmetric-hyperbolic systems. 
In particular, \emph{a unique 
time-continuous solution} can be built 
in $H^s(\RR^3)$ for $t\in[0,T)$,
given initial conditions 
$F^i_\alpha(t=0)\in H^s(\RR^3)^{3\times3}$ and $u^i(t=0)\in H^s(\RR^3)^3$ \cite{DafermosBook4}. 
The latter solution, associated with a unique mapping $\bphi_t$, is equivalently defined 
by 
\cite{wagner-2009}
\begin{align}
\label{eq:momentum_spatial}
& \partial_t \left( \rho u^i \right) +  \partial_j \left( \rho u^i u^j - \sigma^{ij} \right) = \rho f^i
\\
\label{eq:deformation_spatial} 
& \partial_t \left( \rho F^i_\alpha \right) +  \partial_j \left( \rho F^i_\alpha u^j - \rho u^i F^j_\alpha \right) = 0
\\
\label{eq:determinant_spatial} 
& \partial_t \rho + \partial_j\left( \rho u^j \right) = 0 
\\
\label{eq:cofactor_spatial}
& \partial_t \left( \rho C^i_\alpha \right) +  \partial_i \left( \rho C^j_\alpha u^j  \right) = 0 
\end{align}
where 
$\sigma^{ij} := |\bF|^{-1}S^{i\alpha}F^j_\alpha $ 
and $\rho := |\bF|^{-1} \hat\rho$, 
provided the initial conditions satisfy
\begin{equation}
\label{eq:piola_spatial}
\partial_{j}( \rho 
F^j_\alpha ) = 0 = \sigma_{ijk}\partial_j(\rho C^k_\alpha) \quad \forall \alpha \,.
\end{equation}

Indeed, with 
the \emph{Eulerian} description (\ref{eq:momentum_spatial}--\ref{eq:cofactor_spatial}) of the body motions 
(i.e. in spatial coordinates, as opposed to the Lagrangian description
(\ref{eq:momentum_material}--\ref{eq:cofactor_material}) in material coordinates)
\begin{align}
\label{eq:momentum_spatial_tens}
& \partial_t \left( \rho \bu \right) +  \div \left( \rho \bu\otimes\bu - \bsigma \right) = \rho \boldf
\\
\label{eq:deformation_spatial_tens}
& \partial_t \left( \rho \bF \right) -  \grad \times \left( \rho \bF^T\times\bu \right) = \bzero
\\
\label{eq:determinant_spatial_tens} 
& \partial_t \rho + \div \left( \rho \bu \right) = 0 
\\
\label{eq:cofactor_spatial_tens}
& \partial_t \left( \rho \bC \right) +  \grad \otimes \left( \rho \bC^T\cdot\bu \right) = \bzero 
\end{align}
where $\bC^T$ is the dual (matrix transpose) of $\bC$, and with Piola's identity (\ref{eq:piola_spatial})
\begin{equation}
\label{eq:piola_spatial_tens}
\div( \rho \bF^T ) = \bzero = \grad \times (\rho \bC^T) \:,
\end{equation}
one can show that, \emph{when $e(\bF)$ is polyconvex}, 
the symmetric-hyperbolic framework applies to (\ref{eq:momentum_spatial_tens}--\ref{eq:piola_spatial_tens})
insofar as smooth solutions also 
satisfy the 
conservation law
\begin{equation}
\label{eq:energy_spatial}
\partial_t \left( \frac\rho2 |\bu|^2 + \rho e \right) +  
\div \left( \left( \frac\rho2 |\bu|^2 + \rho e \right)\bu - \bsigma\cdot\bu 
\right) = \rho\boldf\cdot\bu
\end{equation}
for $ \frac\rho2 |\bu|^2 + \rho e$, a functional \emph{convex} 
in a set of independent conserved variables \cite{DafermosBook4}.

A first example of a physically-meaningful internal energy is the \emph{neo-Hookean}
\begin{equation}
\label{eq:neohookean}
e(F^k_\alpha F^k_\alpha):=\frac{c_1^2}2 ( F^k_\alpha F^k_\alpha 
-d)
\end{equation} 
with $c_1^2>0$.
Then, the quasilinear system (\ref{eq:momentum_material}--\ref{eq:deformation_material}) is symmetric-hyperbolic 
insofar as smooth solutions additionally satisfy a conservation law for $|\bu|^2/2 + e$ 
strictly convex in 
$(\bu,\bF)$. Unequivocal 
motions can be defined\footnote{
    Not only with fields in $H^s(\RR^3)$, $s>\frac32$, for $t\in[0,T)$,
    but in fact \emph{whatever $T>0$} 
    and $s\in\RR$ 
    here, insofar as $S^i_\alpha(\bF) = \hat\rho c_1^2 
    F^i_\alpha$ 
    so the Lagrangian description (\ref{eq:momentum_material}--\ref{eq:deformation_material})
    reduces to \emph{linear} PDEs. 
}, 
equivalently by (\ref{eq:momentum_spatial_tens}--\ref{eq:deformation_spatial_tens}).
%
%
The latter neo-Hookean model satisfyingly predicts the small motions of some solids. 

However, 
\eqref{eq:neohookean} is oversimplistic 
: it does not model 
the deformations that are often observed 
orthogonally to a stress applied unidirectionally, see e.g. \cite{treloar2005physics} regarding rubber.
Many observations 
are better fitted when the Cauchy stress $\bsigma$ contains an additional spheric term $-p\bI$,
with a pressure $p(\rho)$ function of volume changes.

Next, instead of \eqref{eq:neohookean}, one can rather assume a \emph{compressible neo-Hookean} energy
\begin{equation}
\label{eq:neohookeancomp}
e(F^k_\alpha F^k_\alpha):=\frac{c_1^2}2 
( F^k_\alpha F^k_\alpha 
-d) - \frac{d_1^2}{1-\gamma} |\bF|^{1-\gamma} =: \tilde e(|\bF|,\bF) \,. 
\end{equation} 
The 
functional (\ref{eq:neohookeancomp}) is polyconvex 
as soon as $\gamma>1$ \cite{DafermosBook4}. 
Thus, using either (\ref{eq:momentum_material}--\ref{eq:cofactor_material}) 
or (\ref{eq:momentum_spatial}--\ref{eq:cofactor_spatial}) 
one can define unequivocal smooth motions 
with 
$S^i_\alpha(\bF) = \hat\rho c_1^2 F^i_\alpha - \hat\rho d_1^2 |\bF|^{-\gamma} \mathop{\rm Cof(\bF)}^i_\alpha$
where an additional pressure term arises\footnote{
 And thus the flux becomes nonlinear in the conservative variables, 
 so $T>0$ is definitely finite.
} in comparison with \eqref{eq:neohookean}.
Precisely, one can build unique solutions to a \emph{symmetric} reformulation of a 
system of conservation laws for conserved variables $U(t,\bx):\R^+\times\RR^m\to\RR^n$ i.e.
\begin{equation}
\label{eq:sys}
\partial_tU + \partial_\alpha G_\alpha(U) = 0
\end{equation}
with $k$ involutions $M_\alpha \partial_\alpha U=0$,
$M_\alpha\in\R^{k\times n}$ 
i.e. $M_\alpha G_\beta(U) = - G_\alpha(U) M_\beta$, 
$\alpha\neq\beta$. 

An additional conservation law $\partial_t\eta(U) + \partial_\alpha Q_\alpha(U) = 0$ is satisfied by 
\eqref{eq:sys},
for $\eta(U)=\frac{|\bu|^2}2+\tilde e(|\bF|,\bF)$ a strictly convex functional of $U$.
So a smooth function $\Xi(U)\in\R^{k}$ exists such that $DQ_\alpha(U) = D\eta(U) DG_\alpha(U) + \Xi(U)^TM_\alpha$ holds,
$D^2\eta(U) DG_\alpha(U) + D\Xi(U)^TM_\alpha$ is a symmetric matrix, and 
\eqref{eq:sys} admits a symmetric-hyperbolic reformulation. The 2D 
Lagrangian case 
$\alpha\in\{a,b\}$, $c_1^2 \hat\rho\equiv 1$, reads
\begin{align}
& \partial_t u^x+\partial_a\left(F^y_b p - F^x_a \right)+\partial_b\left(-F^y_a p - F^x_b\right) = 0,
\\
& \partial_t u^y+\partial_a\left(-F^x_b p - F^y_a \right)+\partial_b\left(F^x_a p - F^y_b\right) = 0,
\\
& \partial_t |\bF| = 
\partial_a\left(-F^x_bu^y+F^y_bu^x\right) + \partial_b\left(-F^y_au^x+F^x_au^y\right), 
\\
& \partial_t F^x_a-\partial_au^x = 0,\\  
& \partial_t F^x_b-\partial_bu^x = 0,\\
& \partial_t F^y_a-\partial_au^y = 0,\\  
& \partial_t F^y_b-\partial_bu^y = 0,
\end{align}
with 
$p(|\bF|) := - \partial_{|\bF|}\tilde e \equiv \frac{d_1^2}{c_1^2} |\bF|^{-\gamma}$, 
abusively denoting $(\bu,|\bF|,\bF)$ 
the functions  $(\bu,|\bF|,\bF)\circ\bphi_t$ of material coordinates as usual. 
Involutions $M_\alpha \partial_\alpha U=0$ hold with
$$
M_a = 
\begin{pmatrix}
 0 & 0 & 0 & 0 & 1 & 0 & 0 \\
 0 & 0 & 0 & 0 & 0 & 0 & 1
\end{pmatrix}
\quad
M_b = 
\begin{pmatrix}
 0 & 0 & 0 &-1 & 0 & 0 & 0 \\
 0 & 0 & 0 & 0 & 0 &-1 & 0
\end{pmatrix}
\;.
$$
They combine with \eqref{eq:sys} using
$\Xi(U)^T
=\begin{pmatrix}
p u^y & -p u^x
\end{pmatrix}
$
to yield a symmetric 
system after premultiplication by $D^2\eta(U)$: note $\nu_\alpha(D^2\eta(U) DG_\alpha(U) + D\Xi(U)^TM_\alpha)$ 
reads
\begin{equation}
\label{eq:matrix}
\begin{pmatrix}
 0 & 0 & (\be_x\bC\nu)\partial_{|\bF|}p & -\nu_a & -\nu_b & 0 & 0 \\
 0 & 0 & (\be_y\bC\nu)\partial_{|\bF|}p & 0 & 0 & -\nu_a & -\nu_b \\
(\be_x\bC\nu)\partial_{|\bF|}p & (\be_y\bC\nu)\partial_{|\bF|}p & 0 
& 0 & 0 & 0 & 0 \\
-\nu_a & 0 & 0 & 0 & 0 & 0 & 0 \\
-\nu_b  & 0 & 0 & 0 & 0 & 0 & 0  \\
 0 &-\nu_a & 0 & 0 & 0 & 0 & 0 \\
 0 &-\nu_b & 0 & 0 & 0 & 0 & 0
\end{pmatrix}
\end{equation}
denoting $\be_x\bC\nu \equiv F^y_b\nu_a-F^y_a\nu_b$, $\be_y\bC\nu \equiv -F^x_b\nu_a+F^x_a\nu_b$
and $\nu^T=(\nu_a \; \nu_b)\in\R^m$ a unit vector.
The symmetric formulation allows one to establish the key energy estimates in the existence proof of smooth solutions
\cite{DafermosBook4}, 
as well as weak 
self-similar solutions to the 1D Riemann problem
using generalized eigenvectors $R$ solutions to 
\begin{equation}
\label{eq:eig}
\nu_\alpha \left(D^2\eta(U) DG_\alpha(U) + D\Xi(U)^TM_\alpha\right)R = \sigma D^2\eta(U) R
\end{equation}
with 
eigenvalues $\sigma\in\{0,\pm1,\pm\sqrt{1+ (|\be_x\bC\nu|^2+|\be_y\bC\nu|^2)\partial_{|\bF|}p }\}$. 
For application to real materials\footnote{
 So far, the only parameters to be specified for real application are $\hat\rho$, 
 $c_1^2$ and $d_1^2$. 
}, one important question remains: how to choose 
$c_1^2$ and $d_1^2$. 

In most real applications of elastrodynamics, the \emph{material parameters} $c_1^2$ and $d_1^2$ should vary,
as functions of $\bF$ e.g., but also as functions of an additional \emph{temperature} variable
so as to take into account microscopic 
processes not described by the macroscopic elastodynamics system.
For instance, the deformations endured by 
stressed elastic solids increase with temperature,
until 
the materials become viscous liquids. 
Then, one natural question arises: could \eqref{eq:neohookeancomp} remain useful for liquids
which are mostly incompressible (i.e. $\div\bu\approx0$ holds) and much less elastic than solids~?

In Sec.~\ref{sec:fluid}, 
we recall the limit case when the volumic term dominates the internal energy,
and 
$p 
= 
C_0 \rho^{\gamma}$ 
dominates 
$\bsigma$, 
which coincides with seminal PDEs for \emph{perfect fluids} (fluids without viscosity).
In Section~\ref{sec:maxwell}, we next consider how to rigorously connect fluids like
liquids to solids using an enriched elastodynamics system. 

\subsection{
Fluid 
dynamics} 
\label{sec:fluid}

Consider the general Eulerian description (\ref{eq:momentum_spatial_tens}--\ref{eq:cofactor_spatial_tens})
for continuum body motions.
It is noteworthy that given $\bu$, each \emph{kinematic} equation \eqref{eq:cofactor_spatial},
\eqref{eq:deformation_spatial} and \eqref{eq:determinant_spatial} is autonomous.
As a consequence, \emph{in spatial coordinates}, motions 
can be defined by \emph{reduced versions} 
of the full Eulerian description (\ref{eq:momentum_spatial}--\ref{eq:cofactor_spatial}),
with an internal energy $e$ strictly convex in $\rho$ 
but not 
in $\bF$~!
One famous case  
is the \emph{polytropic law}
\beq
\label{eq:polytropic}
e(\rho) 
:= 
\frac{C_0}{\gamma-1} \rho^{\gamma-1}
\eeq
with 
$C_0>0$.
Then, one obtains Euler's system 
for perfect (inviscid) fluids
\beq
\label{eq:firstorder_spatial_reduced}
\begin{aligned}
& \partial_t \rho + \partial_i( u^i \rho ) = 0 
\\
& \rho \left( \partial_t u^i + u^j \partial_j u^i \right) + \partial_i \: p = \rho f^i
\end{aligned}
\eeq
with a pressure $p:=-\partial_{\rho^{-1}}e=C_0 \rho^\gamma$ 
characterizing \emph{spheric} stresses: 
\beq
\label{eq:cauchystress_pressure}
\sigma^{ij} 
= -p\: \delta_{ij} \,.
\eeq 
The system \eqref{eq:firstorder_spatial_reduced} is symmetric-hyperbolic.
It 
is useful 
to define unequivocal time-evolutions of Eulerian fields (on finite time ranges) \cite{DafermosBook4}, 
although multi-dimensional 
solutions are then not equivalently described by one well-posed 
Lagrangian description \cite{DespresMazeran2005}.
In fact, for application to real fluids, the system \eqref{eq:firstorder_spatial_reduced} is 
better understood as the limit of a kinetic model based on Boltzmann's statistical description of 
molecules \cite{MR2182829}, and the model indeed describes gaseous 
fluids better than condensed fluids 
(liquids). 
In any case, the fluid model \eqref{eq:firstorder_spatial_reduced} still lacks viscosity.

One classical approach 
adds 
viscous stresses as an \emph{extra-stress} term $\btau$ in 
\eqref{eq:cauchystress_pressure} i.e. 
\beq
\label{eq:extra}
\bsigma=-p{\boldsymbol\delta}+\btau \,. 
\eeq
The extra-stress is required symmetric (to preserve angular momentum), 
objective (for the sake of Galilean invariance),
and ``dissipative'' (to satisfy thermodynamics principles) \cite{Coleman-Noll1963}.
Precisely, 
introducing the entropy $\eta$ as an additional 
state variable 
for 
heat exchanges 
at temperature $\theta=\partial_s e>0$, thermodynamics 
requires
$$
\partial_t \eta + ( u^j \partial_j ) \eta = \mathcal{D}/\theta 
$$
with a \emph{dissipation} term $\mathcal{D} \ge 0$.
Usually, denoting $D(u)^{ij} := \frac12\left(\partial_i u^j + \partial_j u^i\right)$, 
one then postulates a \emph{Newtonian} extra-stress 
with two constant parameters $\ell,\dot{\mu}>0$
\beq
\label{eq:extranewt}
\tau^{ij} = 2\dot{\mu} D(u)^{ij} + \ell \: D(u)^{kk} \: \delta_{ij}
\eeq
which satisfies 
$\mathcal{D} \equiv \tau^{ij} \partial_j u^i \ge0 $ 
\cite{Coleman-Noll1963}. 
The  Newtonian model
allows for the definition of causal motions 
through the resulting 
Navier-Stokes equations. But it is not obviously unified with elastodynamics; 
and letting alone 
that \eqref{eq:extranewt} is 
far from some real ``non-Newtonian'' materials,
it implies that shear waves 
propagate infinitely-fast, an idealization that is also a difficulty for the unification with elastodynamics.

By contrast, Maxwell's viscoelastic fluid models for $\btau$ possess well-defined shear waves of finite-speed, 
and they can be connected with elastodynamics 
with a view to unifying solids and fluids (liquids) in a single continuum description.


\section{Viscoelastic flows 
with Maxwell fluids}
\label{sec:maxwell}

Maxwell's 
models \cite{maxwell-1874} with viscosity $\dot\mu>0$, relaxation time $\lambda>0$, time-rate $\stackrel{\Diamond}{\btau}$
\begin{equation}
\label{eq:UCMmodified} 
\lambda \stackrel{\Diamond}{\btau} 
+ \btau = 2 \dot{\mu}  \bD(\bu)
\end{equation}
are widely recognized as physically useful to link fluids 
where $\btau \xrightarrow{\lambda\to0} 2 \dot{\mu}  \bD(\bu)$ in the Newtonian limit, 
with solids governed by elastodynamics when $\lambda\sim\dot{\mu}\to\infty$.
In particular, one often considers the Upper-Convected Maxwell (UCM) model, with objective time-rate 
$\stackrel{\Diamond}{\btau}$ in \eqref{eq:UCMmodified} 
defined 
by the Upper-Convected (UC) derivative\footnote{
 Other objective derivatives than UC can be used, which also allow symmetric-hyperbolic reformulations.
 They will not be considered here for the sake of simplicity.
}:
\begin{equation}
\label{eq:UC}
\stackrel{\triangledown}{\btau} := \partial_t \btau + (\bu\cdot\grad) \btau - (\gbu)\btau  - \btau(\gbu)^T
\end{equation} 
because $\stackrel{\triangledown}{\btau}=2\frac{\dot{\mu}}{\lambda}\bD(\bu)$
is compatible with elastodynamics when 
$\btau=\frac{\dot{\mu}}{\lambda}\left(\bF\bF^T-\bI\right)$.

However, a 
difficulty arises with the quasilinear 
system
\eqref{eq:momentum_spatial_tens}--\eqref{eq:determinant_spatial_tens}--\eqref{eq:extra}--\eqref{eq:UCMmodified}--\eqref{eq:UC}
to define general \emph{multi-dimensional} motions for any $\lambda\in(0,\infty)$ from solutions to Cauchy problems: 
the system may not be hyperbolic 
and numerical simulations may become unstable \cite{owens-philips-2002}.
As a cure,
we proposed in \cite{Boyaval2021} a symmetric-hyperbolic reformulation of 
\eqref{eq:momentum_spatial_tens}--\eqref{eq:determinant_spatial_tens}--\eqref{eq:extra}--\eqref{eq:UCMmodified}--\eqref{eq:UC}
using a new variable $\bA$ in $\btau=\rho(\bF\bA\bF^T-\bI)$.

We review the reformulation in Sec.~\ref{sec:new}, 
after 
recalling in Sec.~\ref{sec:viscoelastic}
well-known 1D solutions to \eqref{eq:momentum_spatial_tens}--\eqref{eq:determinant_spatial_tens}--\eqref{eq:extra}--\eqref{eq:UCMmodified}--\eqref{eq:UC}
which show the interest for Maxwell's models.

\subsection{Viscoelastic 1D shear waves for solids and fluids 
} 
\label{sec:viscoelastic}

Some particular solutions to \eqref{eq:momentum_spatial_tens}--\eqref{eq:determinant_spatial_tens}--\eqref{eq:extra}--\eqref{eq:UCMmodified}--\eqref{eq:UC}
unequivocally model viscoelastic flows, and rigorously link solids to fluids. 
%
Shear waves 
e.g. for a 2D body 
moving along $\be_x\equiv \be_{x^1}$ 
following $b=y\equiv x^2$, $a=x-X(t,y)$, $X(0,y)=0$ 
are well-defined by \eqref{eq:momentum_spatial} 
i.e.
\begin{equation}
\label{eq:momentum1D}
\partial_t u = \partial_y \tau^{xy} 
\end{equation}
where we recall $u:=\partial_t X$, and Maxwell's constitutive relation \eqref{eq:UCMmodified} i.e.
\begin{equation}
\label{eq:maxwell1D}
\lambda \partial_t \tau^{xy} + \tau^{xy} = \dot{\mu} \partial_y u \,, 
\end{equation}
given enough initial and boundary conditions.
Denoting $G:=\frac{\dot{\mu}}\lambda>0$ the shear elasticity, \eqref{eq:momentum1D}--\eqref{eq:maxwell1D} indeed coincides with the famous hyperbolic
system for 1D damped waves, which implies
$\lambda \partial^2_{tt} u(t,y) + \partial_t u(t,y) =  \dot{\mu} \partial^2_{yy} u(t,y)$ 
and
$ \lambda \partial^2_{tt} \tau^{xy}(t,y) + \partial_t \tau^{xy}(t,y) =  \dot{\mu} \partial^2_{yy} \tau^{xy}(t,y) $. 
Time-continuous solutions to \eqref{eq:momentum1D}--\eqref{eq:maxwell1D} are well defined given initial conditions %
plus possibly boundary conditions when the body has finite dimension along $\be_y\equiv \be_{x^2}$,
such as $y \equiv {x^2} >0$ 
in Stokes first problem see e.g. \cite{PREZIOSI1987239}.
Moreover, the latter 1D shear waves rigorously unify solids and fluids insofar as they are \emph{structurally stable} \cite{Payne2001,HYP2022Boyaval}:
when $\lambda\equiv \frac1G \dot{\mu}\to \infty$, they satisfy
$$
\partial^2_{tt} \tau^{xy} = G \partial^2_{yy} \tau^{xy} \quad \partial^2_{tt} u = G \partial^2_{yy} u 
$$
like elastic solids, and when $\lambda\to0$, they satisfy 
$$
\tau^{xy} = \dot{\mu} \partial_{y} u \quad \partial_t u = \dot{\mu} \partial^2_{yy} u 
$$
like viscous liquids. 
So the 1D shear waves illustrate well the structural capability of Maxwell's model to unify solid and Newtonian fluid motions.

But a problem arises with \emph{multi-dimensional} motions: 
solutions to \eqref{eq:momentum_spatial_tens}--\eqref{eq:determinant_spatial_tens}--\eqref{eq:extra}--\eqref{eq:UCMmodified}--\eqref{eq:UC} are not well-defined in general.

\subsection{Maxwell flows with a symmetric-hyperbolic formulation} 
\label{sec:new}

To establish multi-dimensional motions satisfying \eqref{eq:UCMmodified},
we introduced in \cite{Boyaval2021} a 2-tensor $\bA$: 
\begin{equation}
\label{eq:Aeul}
\lambda (\partial_t +\bu\cdot\grad)\bA  + \bA = \bF^{-1}\bF^{-T}
\end{equation}
which can be understood as a 
material 
property that relaxes in fluid flows.
\begin{proposition}
\label{prop:UCMmodified}
 Set $\dot\mu=\lambda c_1^2$. Then $\btau:= \rho c_1^2(\bF\bA\bF^T-\bI)$ satisfies \eqref{eq:UCMmodified} with
    \begin{equation}
    \label{eq:LC}
    \stackrel{\Diamond}{\btau} :=
     \partial_t \btau + (\bu\cdot\grad) \btau - (\gbu)\, \btau  - \btau\, (\gbu)^T 
    + (\div\bu) \,\btau \,.
    \end{equation}
\end{proposition}
\begin{proof}
 Recall that $(\partial_t +\bu\cdot\grad)\bF^T=\bF^T\cdot(\gbu)^T$ 
 holds, using \eqref{eq:deformation_spatial} and \eqref{eq:piola_spatial}.
 Then compute $(\partial_t +\bu\cdot\grad)\btau$ straightforwardly using $\btau:= \rho c_1^2(\bF\bA\bF^T-\bI)$.
\end{proof}

Noteworthily (\ref{eq:UCMmodified})-(\ref{eq:LC}) coincides with a version of Maxwell's models for compressible fluids \cite{Bollada2012}.
Moreover, it is contained in a larger 
symmetric-hyperbolic system, which allows one to rigorously define viscoelastic motions unequivocally.

\begin{proposition} 
With (\ref{eq:extra}) such that $\btau:= \rho c_1^2(\bF\bA\bF^T-\bI)$ 
and $p(\rho)+c_1^2 \rho=-\partial_{\rho^{-1}}e_0$ for 
$e_0$ strictly convex in $\rho^{-1}$, 
\eqref{eq:momentum_spatial_tens}--\eqref{eq:deformation_spatial_tens}--\eqref{eq:determinant_spatial_tens}--\eqref{eq:Aeul}
becomes 
symmetric-hyperbolic provided $\div(\rho\bF^T)=0$ and 
$\bA$ is symmetric positive-definite ($\bA \in \bS^3_{+,*}$).
\end{proposition}
\begin{proof}
 Using $\div(\rho\bF^T)=\bzero$,
 \eqref{eq:momentum_spatial_tens}--\eqref{eq:deformation_spatial_tens}--\eqref{eq:determinant_spatial_tens}--\eqref{eq:Aeul}
 rewrites in material 
 coordinates as 
 the Lagrangian system \eqref{eq:momentum_material}--\eqref{eq:deformation_material}--\eqref{eq:determinant_material}
 plus $\lambda \partial_t \bA  + \bA = \bF^{-1}\bF^{-T}$ where 
 $$
 \bS=(p(|\bF|)\bC + c_1^2 
 \bF^{-T}) + \hat\rho c_1^2 \bF\bA = \hat\rho \partial_{\bF} \left( e_0 + \frac{c_1^2}2 \bF\bA:\bF \right) \,.
 $$
 Then, $\bA \in \bS^3_{+,*}$ allows the variable change $\bY=\bA^{-2}$. The resulting Lagrangian system 
 for $(\bu,\bF,|\bF|,\bY)$ 
 with involution $\grad_{\ba} \times \bF = \bzero$
 admits a 
 ``mathematical entropy'' \cite{Dafermos2013} so 
 it is therefore symmetric-hyperbolic.
 For details we refer to \cite{Boyaval2021}.
\end{proof}

A unique smooth solution can be 
constructed for 
\eqref{eq:momentum_spatial_tens}--\eqref{eq:deformation_spatial_tens}--\eqref{eq:determinant_spatial_tens}--\eqref{eq:Aeul}
using an initial condition satisfying $\rho |\bF|=:\hat\rho >0$, $\div(\rho\bF^T)=0$, $\bA \in \bS_{+,*}$
\cite{DafermosBook4}. 
On small time ranges, it unequivocally defines viscoelastic \emph{multi-dimensional} motions governed by the compressible UCM law
(\ref{eq:UCMmodified})-(\ref{eq:LC}) as long as hyperbolicity holds and the solution remains bounded.
Those 
motions satisfy thermodynamics 
with 
\begin{equation}
\label{eq:energy}
 e = e_0 + \frac{c_1^2}2 \left( \bF\bA:\bF - \log \det \bF\bA:\bF \right) \,. 
\end{equation}

\begin{proposition} 
With (\ref{eq:extra}), $\btau:= \rho c_1^2(\bF\bA\bF^T-\bI)$ 
and $p(\rho)+c_1^2 \rho=-\partial_{\rho^{-1}}e_0$, smooth solutions to 
\eqref{eq:momentum_spatial_tens}--\eqref{eq:deformation_spatial_tens}--\eqref{eq:determinant_spatial_tens}--\eqref{eq:Aeul}
additionally satisfy
\begin{equation*}
\label{eq:energy_spatial_UCM}
\partial_t \left( \frac\rho2 |\bu|^2 + \rho e \right) +  
\div \left( \left( \frac\rho2 |\bu|^2 + \rho e \right)\bu - \bsigma\cdot\bu 
\right) = \rho\boldf\cdot\bu + \frac{\rho c_1^2}{2\lambda}(\bI-\bc^{-1}):(\bc-\bI)
\end{equation*}
provided $\div(\rho\bF^T)=0$ and $\bA \in \bS_{+,*}$,
on denoting $\bc=\bF\bA\bF^T\in \bS_{+,*}$ .
\end{proposition}
\begin{proof}
We will show \eqref{eq:energy_spatial_UCM} in material coordinates (the Lagrangian description).
On one hand, computing $\partial_t |\bu|^2 = 2 \bu\cdot\partial_t\bu$ is straightforward. 
One the other hand, using \eqref{eq:momentum_material} 
and $\partial_t\bF=\grad_{\ba}\bu$ one computes
\begin{multline}
\label{eq:energyestimate}
 \partial_t e = \partial_t e_0 + \frac{c_1^2}2(\bI-\bc^{-1}):\partial_t\bc
  = - \frac{\rho c_1^2}{2\lambda}(\bI-\bc^{-1}):(\bc-\bI) + \grad_{\ba}\bu:\bS/\hat\rho
\end{multline}
where 
$(\bI-\bc^{-1}):(\bc-\bI)\ge0$ is a dissipation. QED
\end{proof}
Interestingly, notice that our 
free energy \eqref{eq:energy} is not useful for well-posedness: 
it is not strictly convex in conserved variables.
Morover, our formulation \eqref{eq:momentum_spatial_tens}--\eqref{eq:deformation_spatial_tens}--\eqref{eq:determinant_spatial_tens}--\eqref{eq:Aeul}
for a sound Maxwell model admits the 1D shear waves examined in Sec.~\ref{sec:viscoelastic} as solutions,
so it preserves some well-established 
interesting properties of the standard (incompressible) formulation of Maxwell model.

Let us finally detail the symmetric structure of our hyperbolic formulation for (compressible) viscoelastic flows
of Maxwell-type, with Lagrangian description 
\begin{align}
\partial_t \bu = \div_{\ba} \bS + \boldf
\\
\partial_t |\bF| = 
\div_{\ba} \left( \bC^T \bu \right) 
\\
\partial_t \bC^T =  \nabla_{\ba} \times \left( \bu\times\bF \right) 
\\
\partial_t \bF^T = \nabla_{\ba} \otimes \bu
\\
\partial_t \bA = ( \bF^{-1}\bF^{-T} - \bA ) /\lambda 
\end{align}
where $\bS=- p\:\bC+ \bF\bA$, $p(|\bF|) = |\bF|^{-1} + \frac{d_1^2}{c_1^2} |\bF|^{-\gamma}$,
assuming $c_1^2 \hat\rho\equiv 1$ in \eqref{eq:energy}
$$
e(\bF)=\frac{c_1^2}2 ( F^k_\alpha A^{\alpha\beta} F^k_\beta - 2 \log |F^k_\beta|  ) - \frac{d_1^2}{1-\gamma} |\bF|^{1-\gamma} \,.
$$
To that aim, we consider a 2D system 
when $\lambda\to\infty$:
\begin{align}
\label{eq:start}
& \partial_t u^x+\partial_a\left(F^y_b p -(A^{aa}F^x_a+A^{ab}F^x_b)\right)+\partial_b\left(-F^y_a p -(A^{ab}F^x_a+A^{bb}F^x_b)\right) = 0,
\\
& \partial_t u^y+\partial_a\left(-F^x_b p -(A^{aa}F^y_a+A^{ab}F^y_b)\right)+\partial_b\left(F^x_a p -(A^{ab}F^y_a+A^{bb}F^y_b)\right) = 0,
\\
& \partial_t |\bF| = 
\partial_a\left(-F^x_bu^y+F^y_bu^x\right) + \partial_b\left(-F^y_au^x+F^x_au^y\right), 
\\
& \partial_t F^x_a-\partial_au^x = 0,\\  
& \partial_t F^x_b-\partial_bu^x = 0,\\
& \partial_t F^y_a-\partial_au^y = 0,\\  
& \partial_t F^y_b-\partial_bu^y = 0,\\
& \partial_t Y^{aa}  = \partial_t Y^{ab} =  \partial_t Y^{bb} = 0
\label{eq:end}
\end{align}
where, 
denoting $\Delta=Y^{aa}Y^{bb}-Y^{ab}Y^{ab}$, $\delta = \sqrt{Y^{aa}+Y^{bb}+2\sqrt{\Delta}}$, we have
$$
A^{aa} = \frac{Y^{bb}+\sqrt{\Delta}}\delta\,,\quad
A^{ab} = \frac{-Y^{ab}}\delta\,,\quad
A^{bb} = \frac{Y^{bb}+\sqrt{\Delta}}\delta\,.
$$ 
Rewriting $\partial_tU + \partial_\alpha G_\alpha(U) = 0$ the system above, 
involutions $M_\alpha \partial_\alpha U=0$ hold with 
$$
M_a = 
\begin{pmatrix}
 0 & 0 & 0 & 0 & 1 & 0 & 0 & 0 & 0 & 0 \\
 0 & 0 & 0 & 0 & 0 & 0 & 1 & 0 & 0 & 0 \\
\end{pmatrix}
\;
M_b = 
\begin{pmatrix}
 0 & 0 & 0 &-1 & 0 & 0 & 0 & 0 & 0 & 0 \\
 0 & 0 & 0 & 0 & 0 &-1 & 0 & 0 & 0 & 0 \\
\end{pmatrix}
$$
and
$\partial_t\eta(U) + \partial_\alpha Q_\alpha(U) = 0$ is satisfied for $\eta(U)=\frac{|\bu|^2}2+e$,
using $\Xi(U)^T
=\begin{pmatrix}
p u^y & -p u^x
\end{pmatrix}
$ in $DQ_\alpha(U) = D\eta(U) DG_\alpha(U) + \Xi(U)^TM_\alpha$.

A symmetric formulation is obtained for our quasilinear formulation of Maxwell (compressible) viscoelastic flows
similarly to the standard compressible elastodynamics case:
on premultiplying the system (\ref{eq:start}--\ref{eq:end}) by $D^2\eta(U)$, insofar as the matrix
$(D^2\eta(U) DG_\alpha(U) + D\Xi(U)^TM_\alpha)\nu_\alpha$ is symmetric given a unit vector $\nu=(\nu_a,\nu_b)\in\R^2$.
We do not detail the symmetric matrix $(D^2\eta(U) DG_\alpha(U) + D\Xi(U)^TM_\alpha)\nu_\alpha$ here:
its upper-left block coincides with \eqref{eq:matrix},
but the other blocks are complicate and depend on the choice of the variable $\bY=\bA^{-\frac12}$ (key to exhibit the symmetric-hyperbolic structure using a fundamental convexity result from \cite{lieb-1973} -- Theorem 2 p.276 with $r=\frac12$ and $p=0$) a choice which is not unique (ours may not be optimal).
In any case, the symmetric structure yields a key energy estimate for the construction of unique smooth solutions,
and it also allows one to construct 1D waves 
similarly from 
\eqref{eq:eig} 
when $\lambda\to\infty$ (otherwise one has to take into account the source term of relaxation-type). 


\section{Conclusion and Perpsectives}

Our symmetric-hyperbolic formulation of viscoelastic flows of Maxwell type \cite{Boyaval2021}
allows one to rigorously establish \emph{multidimensional} motions, within the same continuum perspective as elastodynamics
and Newtonian fluid models.
It remains to exploit that mathematically sound framework, e.g. to establish the structural stability of the model
and rigorously unify (liquid) fluid and solid motions through parameter variations in our model:
see \cite{HYP2022Boyaval} 
regarding the nonsingular limit toward elastodynamics.
Another step in that direction is to drive the transition between (liquid) fluid and solid motions more physically,
e.g. on taking into account heat transfers: 
see \cite{S0219891622500096} for a model of Cattaneo-type for the heat flux,
which preserves the symmetric-hyperbolic structure.
%
%
%
%
%
Last, one may want to add physical effects for particular applications: 
the purely Hookean internal energy in \eqref{eq:energy} can be modified
to include finite-extensibility effects as in FENE-P or Gent models, 
or to use another measure of strain, with lower-convected time-rate for instance, see \cite{S0219891622500096}. 

\section*{Acknowledgment} 

The author acknowledges the support of ANR project 15-CE01-0013 SEDIFLO.

He dedicates this work to the memory of Roland Glowinski, an inspiring pioneer.

\bibliographystyle{plain}

\end{document}